\documentclass[12pt]{amsart}
\usepackage{fullpage}

\newtheorem{theorem}{Theorem}
\newtheorem{corollary}[theorem]{Corollary}
\newtheorem{lemma}[theorem]{Lemma}

\theoremstyle{remark}

\begin{document}

\title{There are $k$-uniform cubefree binary morphisms for all $k \geq 0$}

\author{James Currie and Narad Rampersad}

\address{Department of Mathematics and Statistics \\
University of Winnipeg \\
515 Portage Avenue \\
Winnipeg, Manitoba R3B 2E9 (Canada)}

\email{\{j.currie,n.rampersad\}@uwinnipeg.ca}

\thanks{The first author is supported by an NSERC Discovery Grant.}
\thanks{The second author is supported by an NSERC Postdoctoral
Fellowship.}

\subjclass[2000]{68R15}

\date{\today}

\begin{abstract}
A word is cubefree if it contains no non-empty subword of the form $xxx$.
A morphism $h : \Sigma^* \to \Sigma^*$ is $k$-uniform if $h(a)$ has
length $k$ for all $a \in \Sigma$.  A morphism is cubefree if it maps cubefree
words to cubefree words.  We show that for all $k \geq 0$ there exists
a $k$-uniform cubefree binary morphism.
\end{abstract}

\maketitle

\section{Introduction}

A \emph{square} is a non-empty word of the form $xx$, and a \emph{cube} is a
non-empty word of the form $xxx$.  An overlap is a word of the form
$axaxa$, where $a$ is a letter and $x$ is a word (possibly empty).
A word is \emph{squarefree} (resp. \emph{cubefree}, \emph{overlap-free})
if none of its factors are squares (resp. cubes, overlaps).
The construction of infinite squarefree, cubefree, and overlap-free words
is typically done by iterating a suitable morphism.  Uniform morphisms
have particularly nice properties.  In this note we show that for all
$k \geq 0$ there exists a $k$-uniform cubefree binary morphism.

Let $\Sigma^*$ denote the set of all finite words over the alphabet $\Sigma$.
A morphism $h : \Sigma^* \to \Sigma^*$ is $k$-{\it uniform} if $h(a)$ has
length $k$ for all $a \in \Sigma$; it is {\it uniform} if it is $k$-uniform
for some $k$.  A morphism $h : \Sigma^* \to \Sigma^*$ is \emph{squarefree}
(resp. \emph{cubefree}, \emph{overlap-free}) if $h(w)$ is \emph{squarefree}
(resp. \emph{cubefree}, \emph{overlap-free}) whenever $w \in \Sigma^*$ is
\emph{squarefree} (resp. \emph{cubefree}, \emph{overlap-free}).
Squarefree, cubefree, and overlap-free morphisms have been studied extensively
\cite{BEM79,BS93,Bra83,Cro82,Kar83,Ker84,RS99,RW02,RW04,RW07}.

We denote the \emph{Thue--Morse morphism} by $\theta$:
\begin{eqnarray*}
\theta(0)&=&01\\
\theta(1)&=&10.
\end{eqnarray*}
The \emph{Thue--Morse word} is the infinite fixed point of $\theta$:
\[
{\bf t} = \theta^\omega(0) = 0110100110010110 \cdots
\]
It is well-known that the Thue--Morse word is overlap-free \cite{Thu12}.
Moreover, the Thue--Morse morphism is both overlap-free and cubefree
(see \cite{Bra83,Shu00} for even stronger results).  Berstel and S\'e\'ebold
\cite{BS93} gave a remarkable characterization of
overlap-free binary morphisms: namely, that a binary morphism $h$ is
overlap-free if and only if $h(01101001)$ is overlap-free.  Furthermore,
they showed that if $h$ is an overlap-free binary morphism then $h$ is a power 
of $\theta$ (or its complement).  Thus any overlap-free binary morphism
is $k$-uniform where $k$ is a power of $2$.  It is natural to inquire if
cubefree binary morphisms exhibit similar behaviour.  In this case the answer
is no, as we are able to construct uniform binary morphisms of every
length.

For further background material concerning combinatorics on words we
refer the reader to \cite{AS03}.

\section{Main result}

The main result of this note is that for all $k \geq 0$ there exists a
$k$-uniform cubefree binary morphism.  We begin with some preliminary
lemmas.

\begin{lemma}
\label{lem1}
Let $k \geq 4$ be an integer.  Then the Thue--Morse word ${\bf t}$ contains
two distinct words of length $k$ of the form $0y0$ and two distinct words
of length $k$ of the form $0z1$.
\end{lemma}

\begin{proof}
For $k=4,5,6$ the following table gives the required pairs of subwords.

\begin{table}[htbp]
\begin{center}
\begin{tabular}{|l|l|l|}
\hline
$k=4$ & $(0010, 0100)$ & $(0101, 0011)$ \\
$k=5$ & $(00110, 01100)$ & $(01101, 01001)$ \\
$k=6$ & $(001100, 011010)$ & $(001011, 010011)$ \\
\hline
\end{tabular}
\end{center}
\end{table}

Suppose then that $k > 6$.  If $k$ is even, let $k=2r$; otherwise, let
$k=2r-1$.  Suppose inductively that ${\bf t}$ contains two distinct words
$0y0$ and $0y'0$ of length $r$ and two distinct words $0z1$ and $0z'1$of
length $r$.

If $k$ is even then the words $01\theta(y)01$, $01\theta(y')01$,
and $01\theta(z)10$, $01\theta(z')10$ are the desired words of length $k$.
If $k$ is odd then the words $01\theta(y)0$, $01\theta(y')0$,
and $01\theta(z)1$, $01\theta(z')1$ are the desired words of length $k$.
\end{proof}

The proof of the following lemma essentially follows that of
\cite[Lemma~4]{AC04}.

\begin{lemma}
\label{lem2}
Let $k \geq 7$ be an integer. Then ${\bf t}$ contains two distinct subwords
of length $k$ of the form $01x01$ and two distinct subwords of length $k$
of the form $01x10$.
\end{lemma}

\begin{proof}
We only give the details for $01x01$, the proof for $01x10$ being analogous.
If $k$ is even, let $k = 2r$. We have $r = k/2 \geq 4$, so that ${\bf t}$
contains distinct words $u = 0v0$ and $u' = 0v'0$ of length $r$ by
Lemma~\ref{lem1}.  The words $\theta(u) = 01\theta(v)01$ and
$\theta(u') = 01\theta(v')01$ are therefore words of the required form of
length $k$.

If $k$ is odd and $k \geq 23$, we can write $k$ as $8r - 9$, $8r - 7$,
$8r - 5$ or $8r - 3$ for some $r \geq 4$. Let $u = 0v0$ and $u' = 0v'0$
be distinct words of length $r$ in ${\bf t}$. The word
$$
\theta^3(u) = 011\underline{01}001 \theta^3(v) \underline{01}1010\underline{01}
$$
contains words $01x01$ of lengths $8r - 9$ (including the first and second
underlined $01$'s) and $8r - 3$ (including the first and third underlined
$01$'s.)  Similarly, the word
$$
\theta^3(u') = 01101001 \theta^3(v') 01101001
$$
contains words $01x'01$ of lengths $8r - 9$ and $8r - 3$.  Moreover,
since $v \neq v'$, these words are distinct from the corresponding
subwords of $\theta^3(u)$.

Let $z = 0v1$ and $z' = 0v'1$ be distinct words of length $r$ in ${\bf t}$.
The word
$$
\theta^3(z) = 011\underline{01}001 \theta^3(v) 10\underline{01}\underline{01}10
$$
contains words $01x01$ of lengths $8r - 7$ (including the first and second
underlined $01$'s) and $8r - 5$ (including the first and third underlined
$01$'s.)  Similarly, the word
$$
\theta^3(z') = 01101001 \theta^3(v') 10010110
$$
contains words $01x'01$ of lengths $8r - 7$ and $8r - 5$.  Moreover,
since $v \neq v'$, these words are distinct from the corresponding
subwords of $\theta^3(z)$.

For $k$ odd, $7 \leq k \leq 21$, the following table gives the required pairs
of subwords.

\begin{table}[htbp]
\begin{center}
\begin{tabular}{|l|ll|}
\hline
$k=7$ & $0100101$ & $0101101$ \\
$k=9$ & $010011001$ & $011001101$ \\
$k=11$ & $01001100101$ & $01100101101$ \\
$k=13$ & $0100101101001$ & $0110100101101$ \\
$k=15$ & $011001011001101$ & $010011001011001$ \\
$k=17$ & $01001011001101001$ & $01101001011001101$ \\
$k=19$ & $0100101100110100101$ & $0101101001100101101$ \\
$k=21$ & $011010011001011001101$ & $011001101001100101101$ \\
\hline
\end{tabular}
\end{center}
\end{table}
\end{proof}

\begin{lemma}
\label{lem3}
Let $k \geq 9$ be an integer. Then there exist two distinct cubefree words
of length $k$ of the form $00x11$.
\end{lemma}

\begin{proof}
For $9 \leq k \leq 14$, the following table gives the required pairs
of subwords.

\begin{table}[htbp]
\begin{center}
\begin{tabular}{|l|ll|}
\hline
$k=9$ & $001001011$ & $001010011$ \\
$k=10$ & $0010011011$ & $0010110011$ \\
$k=11$ & $00100110011$ & $00101001011$ \\
$k=12$ & $001001010011$ & $001001011011$ \\
$k=13$ & $0010010110011$ & $0010011001011$ \\
$k=14$ & $00100101001011$ & $00100101101011$ \\
\hline
\end{tabular}
\end{center}
\end{table}

Suppose $k \geq 15$.  If $k$ is even, let $k=2r-2$; otherwise, let $k=2r+1$.
Note that $r \geq 7$, so by Lemma~\ref{lem2}, there are distinct subwords
$01x10$ and $01x'10$ of ${\bf t}$ of length $r$.

If $k$ is even, then the complements of the words
$0^{-1}\theta(01x10)1^{-1} = 110\theta(x)100$ and
$0^{-1}\theta(01x'10)1^{-1} = 110\theta(x')100$ are cubefree words of the
desired form of length $k$.

If $k$ is odd, then let $u = 11\theta(01x10)1^{-1} = 110110\theta(x)100$
and $u' = 11\theta(01x'10)1^{-1} = 110110\theta(x')100$.  We claim that
$u$ and $u'$ are cubefree.  Suppose to the contrary that $u$ contains a cube.
Since $0110\theta(x)100$ is overlap-free, any such cube would have to start
with either the first or second 1, but in either case, by inspection
the period of the cube is at least 3, which forces an overlap in
$0110\theta(x)100$, a contradiction.  Similarly, $u'$ is cubefree.
Taking the complements of $u$ and $u'$ gives cubefree words of the desired
form of length $k$.
\end{proof}

\begin{theorem}
\label{main_thm}
Let $k \geq 5$ be an integer. Let $w_0$, $w_1\in 00\{0,1\}^k11$ be distinct
cube-free words. The morphism $\phi:\{0,1\}^*\rightarrow\{0,1\}^*$ given by
$$
\phi(i)=\theta(w_i)(010)^{-1}
$$
is cube-free.
\end{theorem}

\begin{proof}
The existence of $w_0$ and $w_1$ is guaranteed by Lemma~\ref{lem3}.
Suppose that $v\in\{0,1\}^*$ is cube-free, but $\phi(v)$
contains a cube $xxx$. Let $p=|x|$. For $i=0,1$, 
since neither $000$ nor $111$ is a factor of $w_i$,  
word $\phi(i)$ cannot have $10101$ as a factor; 
for the same reason, word $\phi(i)$ has prefix $01011$ and suffix $11$. 
Thus $10101$ occurs as a factor in $\phi(v)$ exactly at the boundaries
between images of letters of $v$. It follows that the indices of any
occurences of $10101$ in $\phi(v)$ differ by multiples of $|\phi(0)|$. 
Again, since $10101$ always occurs in $\phi(v)$ in the context $1101011$,
no proper extension of $10101$ in $v$ has period 1, 2, 3 or 4.

Since $w_i$ and $\theta$ are cube-free, for $i\in \{0,1\}$, the word
$\phi(i)010=\theta(w_i)$ is cube-free.
It follows that $xxx$ spans the border between $\phi(i)$ and $\phi(j)$
for some $i$, $j\in\{0,1\}$ and in fact $xxx$ contains factor 10101.
Since $10101$ is cube-free, $xxx$ is a proper extension of $10101$,
and thus has period at least 5. Note that any factor $u$ of
$xxx$ with $|u|\le p$ occurs twice in $xxx$ with indices differing by $p$.
In particular, since $|10101|=5\le p$, two occurrences of $10101$ in $xxx$
have indices differing by $p$. We conclude that $p$ is a multiple of
$|\phi(0)|$. 

Write $x=a\phi(u)b$ where $u\in\{0,1\}^*$, $|ab|=|\phi(0)|$.
We have $xxx=a\phi(u)ba\phi(u)ba\phi(u)b$, and $ba=\phi(i_0)$
for some $i_0\in\{0,1\}$. However, since $w_1\ne w_2$, we also have
$\phi(0)\ne \phi(1)$ so that either
\begin{itemize}
\item{at most one of $\phi(0)$, $\phi(1)$ has $b$ as a prefix OR}
\item{at most one of $\phi(0)$, $\phi(1)$ has $a$ as a suffix.}
\end{itemize}
Suppose that at most one of $\phi(0)$, $\phi(1)$ has $b$ as a prefix.
(The other case is similar.) Without loss of generality, say that
$\phi(0)$ has $b$ as a prefix. It follows that $\phi(v)$ contains
$\phi(u0u0u0)$,  and $v$ contains $u0u0u0$ as a factor. This is a
contradiction.
\end{proof}

\begin{corollary}
For every integer $k \geq 0$, there exists a $k$-uniform cubefree binary
morphism.
\end{corollary}

\begin{proof}
If $k$ is odd and $k \geq 15$, then Theorem~\ref{main_thm} gives a cubefree
morphism of length $k$.  For $k \in \{3,5,7,11,13\}$, the morphisms given in
the table below are cubefree.

\begin{table}[htbp]
\begin{center}
\begin{tabular}{|l|ll|}
\hline
$\phi_3$ & $0 \to 001$ & $1 \to 011$ \\
$\phi_5$ & $0 \to 01001$ & $1 \to 10110$ \\
$\phi_7$ & $0 \to 0010011$ & $1 \to 0011011$ \\
$\phi_{11}$ & $0 \to 00101001011$ & $1 \to 00101001101$ \\
$\phi_{13}$ & $0 \to 0010010110011$ & $1 \to 0010011001011$ \\
\hline
\end{tabular}
\end{center}
\end{table}

The cubefreeness of these morphisms can be established by
a criterion of Ker\"anen \cite{Ker84}, which states that to confirm that a
uniform binary morphism is cubefree, it suffices to check that the
images of all words of length at most 4 are cubefree.

For $k=1$, the identity morphism is certainly cubefree, and for $k=9$,
clearly we may take $\phi_3^2$.  This establishes the result for all odd $k$.

If $k=0$, the morphism that maps every word to the empty word is trivially
cubefree.  If $k$ is positive, even and not a power of 2, then $k = 2^a(2r+1)$
for some positive $a,r$.  If $\phi$ is a $(2r+1)$-uniform cubefree morphism,
then the morphism $\theta^a \circ \phi$ is a $k$-uniform cubefree morphism.
Similarly, if $k = 2^a$, then $\theta^a$ is a $k$-uniform cubefree morphism.
This completes the proof.
\end{proof}

Brandenburg \cite{Bra83} gave an example of an $11$-uniform squarefree
ternary morphism and stated further that there are no smaller uniform
squarefree ternary morphisms (excluding $0$-uniform and $1$-uniform morphisms).
We therefore conclude by asking:
\begin{quotation}
Do there exist $k$-uniform squarefree ternary morphisms for all $k \geq 11$?
\end{quotation}

\bibliographystyle{amsalpha}

\begin{thebibliography}{99}
\bibitem{AC04}
A. Aberkane, J. Currie, ``There exist binary circular $5/2^+$ power free
words of every length'', \textit{Electron. J. Combinatorics} \textbf{11}
(2004), \#R10.

\bibitem{AS03}
J.-P. Allouche, J. Shallit, \textit{Automatic Sequences: Theory, Applications,
Generalizations}, Cambridge, 2003.

\bibitem{BEM79}
D. Bean, A. Ehrenfeucht, G. McNulty, ``Avoidable patterns in strings of
symbols'', \textit{Pacific J. Math.} \textbf{85} (1979), 261--294.

\bibitem{BS93}
J. Berstel, P. S\'e\'ebold, ``A characterization of overlap-free
morphisms'', \emph{Discrete Appl. Math.} \textbf{46} (1993), 275--281.

\bibitem{Bra83}
F.-J. Brandenburg, ``Uniformly growing $k$th power-free homomorphisms'',
\textit{Theoret. Comput. Sci.} \textbf{23} (1983), 69--82.

\bibitem{Cro82}
M. Crochemore, ``Sharp characterizations of squarefree morphisms'',
\textit{Theoret. Comput. Sci.} \textbf{18} (1982), 221--226.

\bibitem{Kar83}
J. Karhum\"aki, ``On cube-free $\omega$-words generated by binary morphisms'',
\textit{Discrete Appl. Math.} \textbf{5} (1983), 279--297.

\bibitem{Ker84}
V. Ker\"anen, ``On $k$-repetition freeness of length uniform morphisms
over a binary alphabet'', \textit{Discrete Appl. Math.} \textbf{9} (1984),
301--305.

\bibitem{RS99}
G. Richomme, P. S\'e\'ebold, ``Characterization of test-sets for overlap-free
morphisms'', \textit{Discrete Appl. Math.} \textbf{98} (1999), 151--157.

\bibitem{RW02}
G. Richomme, F. Wlazinski, ``Some results on $k$-power-free morphisms'',
\textit{Theoret. Comput. Sci.} \textbf{273} (2002), 119--142.

\bibitem{RW04}
G. Richomme, F. Wlazinski, ``Overlap-free morphisms and finite test-sets'',
\textit{Discrete Appl. Math.} \textbf{143} (2004), 92--109.

\bibitem{RW07}
G. Richomme, F. Wlazinski, ``Existence of finite test-sets for
$k$-power-freeness of uniform morphisms'', \textit{Discrete Appl. Math.}
\textbf{155} (2007), 2001--2016.

\bibitem{Shu00}
A.M. Shur, ``The structure of the set of cube-free $\mathbb{Z}$-words in a
two-letter alphabet'' (Russian), \emph{Izv. Ross. Akad. Nauk Ser. Mat.}
\textbf{64} (2000), 201--224.  English translation in \emph{Izv. Math.}
\textbf{64} (2000), 847--871.

\bibitem{Thu12}
A. Thue, ``\"Uber die gegenseitige Lage gleicher Teile gewisser
Zeichenreihen'', \textit{Kra. Vidensk. Selsk. Skrifter. I. Math. Nat. Kl.}
\textbf{1} (1912), 1--67.
\end{thebibliography}

\end{document}